\documentclass{amsart}

\usepackage[english]{babel}
\usepackage{csquotes}
\usepackage[style=alphabetic,
isbn=false,
doi=false,
url=false,
backend=bibtex,
maxnames=10,
maxalphanames=5
]{biblatex}
\renewbibmacro{in:}{}
\bibliography{sources.bib}
\usepackage{amsmath,amsfonts,amsthm,mathtools, amssymb}
\usepackage{xcolor}
\usepackage{tikz-cd}
\usepackage{hyperref}
\usepackage{aliascnt}
\usepackage{xparse}
\usepackage{tensor}
\usepackage{caption}
\usepackage{mathrsfs}
\usepackage{tabto}
\usepackage{imakeidx}
\usepackage[normalem]{ulem}
\usepackage{upgreek}

\usepackage{tikz}
\usetikzlibrary{matrix,arrows,calc}
\usetikzlibrary{positioning}
\usetikzlibrary{decorations.pathreplacing,decorations.markings,decorations.pathmorphing}
\usetikzlibrary{positioning,arrows,patterns}
\usetikzlibrary{cd}
\usetikzlibrary{intersections}

\tikzset{
	every loop/.style={very thick},
	comp/.style={circle,fill,black,,inner sep=0pt,minimum size=5pt},
	order bottom left/.style={pos=.05,left,font=\tiny},
	order top left/.style={pos=.9,left,font=\tiny},
	order bottom right/.style={pos=.05,right,font=\tiny},
	order top right/.style={pos=.9,right,font=\tiny},
	order node dis/.style={text width=.75cm},
	circled number/.style={circle, draw, inner sep=0pt, minimum size=12pt},
	below left with distance/.style={below left,text height=10pt},
	below right with distance/.style={below right,text height=10pt}
}

\makeatletter
\ifcsname phantomsection\endcsname
\newcommand*{\@gobblenexttocentry}[9]{}
\else
\newcommand*{\@gobblenexttocentry}[4]{}
\fi
\newcommand*{\addsubsection}{%
	\addtocontents{toc}{\protect\@gobblenexttocentry}%
	\subsection*}
\makeatother

\begin{document}
	
	\def\subsectionautorefname{Section}
	\def\subsubsectionautorefname{Section}
	\def\sectionautorefname{Section}
	\def\equationautorefname~#1\null{(#1)\null}

\newcommand{\mynewtheorem}[4]{
	\if\relax\detokenize{#3}\relax 
	\if\relax\detokenize{#4}\relax 
	\newtheorem{#1}{#2}
	\else
	\newtheorem{#1}{#2}[#4]
	\fi
	\else
	\newaliascnt{#1}{#3}
	\newtheorem{#1}[#1]{#2}
	\aliascntresetthe{#1}
	\fi
	\expandafter\def\csname #1autorefname\endcsname{#2}
}

\newcommand{\mcomment}[1]{\textcolor{red}{#1}}
\newcommand{\dcomment}[1]{\textcolor{blue}{#1}}

\mynewtheorem{theorem}{Theorem}{}{section}
\mynewtheorem{lemma}{Lemma}{theorem}{}
\mynewtheorem{prop}{Proposition}{lemma}{}
\mynewtheorem{cor}{Corollary}{lemma}{}
\mynewtheorem{question}{Question}{lemma}{}
\mynewtheorem{assumption}{Assumption}{lemma}{}
\mynewtheorem{example}{Example}{lemma}{}
\theoremstyle{definition}
\newtheorem{definition}[theorem]{Definition}

\theoremstyle{remark}
\newtheorem{reminder}[theorem]{Reminder}
\newtheorem{rem}[theorem]{Remark}


\def\defbb#1{\expandafter\def\csname b#1\endcsname{\mathbb{#1}}}
\def\defcal#1{\expandafter\def\csname c#1\endcsname{\mathcal{#1}}}
\def\deffrak#1{\expandafter\def\csname frak#1\endcsname{\mathfrak{#1}}}
\def\defop#1{\expandafter\def\csname#1\endcsname{\operatorname{#1}}}
\def\defbf#1{\expandafter\def\csname b#1\endcsname{\mathbf{#1}}}

\makeatletter
\def\defcals#1{\@defcals#1\@nil}
\def\@defcals#1{\ifx#1\@nil\else\defcal{#1}\expandafter\@defcals\fi}
\def\deffraks#1{\@deffraks#1\@nil}
\def\@deffraks#1{\ifx#1\@nil\else\deffrak{#1}\expandafter\@deffraks\fi}
\def\defbbs#1{\@defbbs#1\@nil}
\def\@defbbs#1{\ifx#1\@nil\else\defbb{#1}\expandafter\@defbbs\fi}
\def\defbfs#1{\@defbfs#1\@nil}
\def\@defbfs#1{\ifx#1\@nil\else\defbf{#1}\expandafter\@defbfs\fi}
\def\defops#1{\@defops#1,\@nil}
\def\@defops#1,#2\@nil{\if\relax#1\relax\else\defop{#1}\fi\if\relax#2\relax\else\expandafter\@defops#2\@nil\fi}
\makeatother

\defbbs{ZHQCNPALRVWBEVW}
\defcals{DOPQMNXYLTRAEHZKCFIUZ}
\deffraks{apijklmnopqueR}
\defops{PGL,SL,Sp,mod,Spec,Re,Gal,Tr,End,GL,Hom,PSL,H,div,Aut,rk,Mod,R,T,Tr,Mat,Vol,MV,Res,vol,Z,diag,Hyp,ord,Im,ev,U,dev,c,CH,fin,pr,Pic,lcm,ch,td,LG,id,Sym,Aut}
\defbfs{kiuvzwp} 

\def\ep{\varepsilon}
\def\abs#1{\lvert#1\rvert}
\def\dd{\mathrm{d}}
\def\inj{\hookrightarrow}
\def\eq{=}
\newcommand{\hyp}{{\rm hyp}}
\newcommand{\odd}{{\rm odd}}

\def\i{\mathrm{i}}
\def\e{\mathrm{e}}
\def\st{\mathrm{st}}
\def\ct{\mathrm{ct}}

\def\uC{\underline{\bC}}
\def\ol{\overline}

\def\Vrel{\bV^{\mathrm{rel}}}
\def\Wrel{\bW^{\mathrm{rel}}}
\def\twolev{\mathrm{LG_1(B)}}

\def\be{\begin{equation}}   \def\ee{\end{equation}}     \def\bes{\begin{equation*}}    \def\ees{\end{equation*}}
\def\ba{\be\begin{aligned}} \def\ea{\end{aligned}\ee}   \def\bas{\bes\begin{aligned}}  \def\eas{\end{aligned}\ees}
\def\={\;=\;}  \def\+{\,+\,} \def\m{\,-\,}

\newcommand*{\proj}{\mathbb{P}}
\newcommand{\barmoduli}[1][g]{{\overline{\mathcal M}}_{#1}}
\newcommand{\moduli}[1][g]{{\mathcal M}_{#1}}
\newcommand{\omoduli}[1][g]{{\Omega\mathcal M}_{#1}}
\newcommand{\modulin}[1][g,n]{{\mathcal M}_{#1}}
\newcommand{\omodulin}[1][g,n]{{\Omega\mathcal M}_{#1}}
\newcommand{\zomoduli}[1][]{{\mathcal H}_{#1}}
\newcommand{\barzomoduli}[1][]{{\overline{\mathcal H}_{#1}}}
\newcommand{\pomoduli}[1][g]{{\proj\Omega\mathcal M}_{#1}}
\newcommand{\pomodulin}[1][g,n]{{\proj\Omega\mathcal M}_{#1}}
\newcommand{\pobarmoduli}[1][g]{{\proj\Omega\overline{\mathcal M}}_{#1}}
\newcommand{\pobarmodulin}[1][g,n]{{\proj\Omega\overline{\mathcal M}}_{#1}}
\newcommand{\potmoduli}[1][g]{\proj\Omega\tilde{\mathcal{M}}_{#1}}
\newcommand{\obarmoduli}[1][g]{{\Omega\overline{\mathcal M}}_{#1}}
\newcommand{\obarmodulio}[1][g]{{\Omega\overline{\mathcal M}}_{#1}^{0}}
\newcommand{\otmoduli}[1][g]{\Omega\tilde{\mathcal{M}}_{#1}}
\newcommand{\pom}[1][g]{\proj\Omega{\mathcal M}_{#1}}
\newcommand{\pobarm}[1][g]{\proj\Omega\overline{\mathcal M}_{#1}}
\newcommand{\pobarmn}[1][g,n]{\proj\Omega\overline{\mathcal M}_{#1}}
\newcommand{\princbound}{\partial\mathcal{H}}
\newcommand{\omoduliinc}[2][g,n]{{\Omega\mathcal M}_{#1}^{{\rm inc}}(#2)}
\newcommand{\obarmoduliinc}[2][g,n]{{\Omega\overline{\mathcal M}}_{#1}^{{\rm inc}}(#2)}
\newcommand{\pobarmoduliinc}[2][g,n]{{\proj\Omega\overline{\mathcal M}}_{#1}^{{\rm inc}}(#2)}
\newcommand{\otildemoduliinc}[2][g,n]{{\Omega\widetilde{\mathcal M}}_{#1}^{{\rm inc}}(#2)}
\newcommand{\potildemoduliinc}[2][g,n]{{\proj\Omega\widetilde{\mathcal M}}_{#1}^{{\rm inc}}(#2)}
\newcommand{\omoduliincp}[2][g,\lbrace n \rbrace]{{\Omega\mathcal M}_{#1}^{{\rm inc}}(#2)}
\newcommand{\obarmoduliincp}[2][g,\lbrace n \rbrace]{{\Omega\overline{\mathcal M}}_{#1}^{{\rm inc}}(#2)}
\newcommand{\obarmodulin}[1][g,n]{{\Omega\overline{\mathcal M}}_{#1}}
\newcommand{\LTH}[1][g,n]{{K \overline{\mathcal M}}_{#1}}
\newcommand{\PLS}[1][g,n]{{\bP\Xi \mathcal M}_{#1}}

\newcommand*{\can}{\mathrm{can}}
\newcommand*{\sing}{\mathrm{sing}}
\newcommand*{\reg}{\mathrm{reg}}
\newcommand*{\Ex}{\mathrm{Ex}}
\newcommand*{\Spin}{\mathrm{Spin}}
\newcommand*{\SU}{\mathrm{SU}}
\newcommand*{\sm}{\mathrm{sm}}
\newcommand*{\codim}{\mathrm{codim}}
\newcommand*{\Proj}{\mathrm{Proj}}
\newcommand*{\ad}{\mathrm{ad}}
\newcommand*{\Div}{\mathrm{Div}}
\newcommand*{\Ad}{\mathrm{Ad}}
\newcommand*{\SO}{\mathrm{SO}}

\newcommand*{\Tw}[1][\Lambda]{\mathrm{Tw}_{#1}}  
\newcommand*{\sTw}[1][\Lambda]{\mathrm{Tw}_{#1}^s}  

\newcommand{\bfa}{{\bf a}}
\newcommand{\bfb}{{\bf b}}
\newcommand{\bfd}{{\bf d}}
\newcommand{\bfe}{{\bf e}}
\newcommand{\bff}{{\bf f}}
\newcommand{\bfg}{{\bf g}}
\newcommand{\bfh}{{\bf h}}
\newcommand{\bfm}{{\bf m}}
\newcommand{\bfn}{{\bf n}}
\newcommand{\bfp}{{\bf p}}
\newcommand{\bfq}{{\bf q}}
\newcommand{\bfP}{{\bf P}}
\newcommand{\bfR}{{\bf R}}
\newcommand{\bfU}{{\bf U}}
\newcommand{\bfu}{{\bf u}}
\newcommand{\bfz}{{\bf z}}

\newcommand{\bfl}{{\boldsymbol{\ell}}}
\newcommand{\bfmu}{{\boldsymbol{\mu}}}
\newcommand{\bfeta}{{\boldsymbol{\eta}}}
\newcommand{\bfomega}{{\boldsymbol{\omega}}}

\newcommand{\wh}{\widehat}
\newcommand{\wt}{\widetilde}

\newcommand{\ps}{\mathrm{ps}}  

\newcommand{\tdpm}[1][{\Gamma}]{\mathfrak{W}_{\operatorname{pm}}(#1)}
\newcommand{\tdps}[1][{\Gamma}]{\mathfrak{W}_{\operatorname{ps}}(#1)}

\newlength{\halfbls}\setlength{\halfbls}{.5\baselineskip}
\newlength{\halbls}\setlength{\halfbls}{.5\baselineskip}

\newcommand*{\Hrel}{\cH_{\text{rel}}^1}
\newcommand*{\Hrelbar}{\overline{\cH}^1_{\text{rel}}}

\newcommand*\interior[1]{\mathring{#1}}

\newcommand{\prodt}[1][\lceil j \rceil]{t_{#1}}
\newcommand{\prodtL}[1][\lceil L \rceil]{t_{#1}}


	
	\title[Milnor-Wood inequality for klt varieties and uniformization]
	{Milnor-Wood inequality for klt varieties of general type and uniformization}
	
	\author{Matteo Costantini}
	\address{Essener Seminar für Algebraische Geometrie und Arithmetik, Fakultät für Mathematik, Universit\"at Duisburg-Essen, 45117 Essen}
	\email{matteo.costantini@uni-due.de}
	
	\author{Daniel Greb}
	\address{Essener Seminar für Algebraische Geometrie und Arithmetik, Fakultät für Mathematik, Universit\"at Duisburg-Essen, 45117 Essen}
	\email{daniel.greb@uni-due.de}
	
	\thanks{The research of the authors has been supported by the DFG-Research Training Group 2553 “Symmetries and classifying spaces: analytic, arithmetic, and derived”.}
	
	\date{\today}
	
	\subjclass[2020]{32Q30, 32M15, 53C35, 14E30, 53C24, 53C43}
	
	\begin{abstract}
		We generalize the definition of the Toledo invariant for representations of fundamental groups of smooth varieties of general type due to Koziarz and Maubon to the context of singular klt varieties, where the natural fundamental groups to consider are those of smooth loci. Assuming that the rank of the target Lie group is not greater than two, we show that the Toledo invariant satisfies a Milnor-Wood type inequality and we characterize the corresponding maximal representations.
	\end{abstract}
	\maketitle
	\tableofcontents
	
\section{Introduction}

The Toledo invariant is classically a characteristic number naturally associated to representations of lattices of semisimple Lie groups of Hermitian type into other semisimple Lie groups of Hermitian type. Recall that a real semisimple Lie group $G$ (with no compact factors) is said to be of Hermitian type if its associated symmetric space is Hermitian symmetric, which means that it admits a $G$-invariant K\"ahler form. The most general classical definition was given in \cite{burgeriozzi} in terms of the second bounded cohomology of the involved Lie groups and lattices. The Milnor-Wood inequality and in particular maximal representations were deeply studied in many works, see for example \cite{Toledo}, \cite{corlette}, \cite{GM},\cite{BGPG}, \cite{BIW}, \cite{Poz}, \cite{km2}, \cite{KMfoliatedhiggs}, \cite{spinaci} only to cite a few. In this paper we are interested in the generalization considered by Koziarz and Maubon in \cite{KMgentype}, in which they no longer considered representations of  complex hyperbolic lattices,  but more generally of fundamental groups of smooth varieties of general type. This generalization was possible since a proof of the Milnor-Wood inequality and the uniformizing equality case can be given in differential geometric terms using the point of view of Higgs bundles via Simpson's correspondence.

Motivated by the fact that ball quotients can have finite quotient  singularities, which are particular examples of klt singularities, and by recent generalizations (\cite{GKPT, grebkebpetmathann}) of classical uniformization results to the klt setup, which naturally appears when considering minimal models of varieties of general type, we investigate representations of fundamental groups associated with singular varieties of general type. It turns out that in the presence of singularities the natural representations to consider are those of the fundamental groups of smooth loci. The definition of Toledo invariant can be made in this more general context thanks to the work of Mochizuki \cite{mochizuki_asterisque}, in which he generalizes the classical result of Corlette  and shows existence of harmonic metrics for flat bundles over quasi-projective varieties. Moreover, the Higgs bundle approach of Koziarz and Maubon to prove the Milnor-Wood inequality can also be  applied in the singular klt setting because of the extension of Simpson's correspondence to this case by Greb-Kebekus-Peternell-Taji \cite{grebkebpetmathann}.

We begin by defining the Toledo invariant; throughout we will work over the complex numbers. Let $X$ be an algebraic klt variety of dimension $d\geq 2$ of general type with nef canonical bundle, cf.~\autoref{sec:klt}. Let $X_{\text{reg}}\subseteq X$ be its regular locus, let $G$ be a Hermitian  Lie group, and let $\rho\colon \pi_1(X_\reg) \to G$ be a reductive representation.\footnote{See Remark~\ref{rem:nonred} for a discussion of non-reductive representations.}
Let $(X_{\text{reg}})^\mathfrak{u}\to X_{\text{reg}}$ be the universal cover of the regular locus and $f:(X_{\text{reg}})^\mathfrak{u}\to \mathcal{Y}_G$ be the harmonic $\rho$-equivariant map to the symmetric space of $G$ whose existence and uniqueness follows from the work of Mochizuki \cite[Prop. A.18]{mochizuki_asterisque}. Let finally $\omega_{\mathcal{Y}_G}$ be the cohomology class of the K\"ahler form of the symmetric metric on $\mathcal{Y}_G$. Then, we define the \emph{Toledo invariant} of the representation as 
\begin{equation}\label{eq:MWinequ}
\tau(\rho):=\frac{1}{4\pi}\int_X f^*(\omega_{\mathcal{Y}_G})\wedge \c_1(K_X)^{d-1} \in \mathbb{Q}
\end{equation}
Here, $f^*(\omega_{\mathcal{Y}_G})$ is considered as a class on $X$ (which we can due owing to $\rho$-equivariance of $f$). Note that the previous definition is well-defined since $K_X$ is big and nef, which implies that some power is base point free, and hence by Bertini's theorem  that we can move a multiple of the class $\c_1(K_X)^{d-1}$ away from the singularities. The fact that the Toledo invariant is a rational number follows from considerations recalled at the end of Section~\ref{sec:Ghiggs} below. 

The following is our main result.
\begin{theorem}
\label{thm:intro}
Let $X$ be a projective variety of general type  of dimension $d\geq 2$ with at worst klt singularities and with nef canonical bundle, and let $X_{\text{reg}}\subseteq X$ be its regular locus. 
Let $\rho:\pi_1(X_{\text{reg}})\longrightarrow G$ be a reductive representation into a classical Hermitian Lie groups of rank $\rk(G)\leq 2$ other than $G=\SO^*(10)$. 

Then, the following inequality holds: 
\[|\tau(\rho)| \leq \rk(G)\cdot\frac{K_{X}^d}{d+1}.\]

Equality holds if and only  if the canonical model $X_{\text{can}}$ of $X$ is the quasi-étale quotient of a smooth ball quotient $Z$ by a finite group, $G=\SU(p,q)$ with $p\geq qd$, and $\rho$ induces a representation $\rho_Z:\pi_1(Z)\to G$ whose $\rho_Z$-equivariant harmonic map is a holomorphic or antiholomorphic proper embedding from the universal cover of $Z$  onto a totally geodesic copy of complex hyperbolic $d$-space, of induced holomorphic sectional curvature $-1/q$, in the symmetric space associated to $G$.
\end{theorem}

The hypothesis on the rank of $G$ should morally not be there, since in the homogeneous situation \cite{KMfoliatedhiggs} it is not a needed as an assumption, and since furthermore the uniformization result in the maximal case does not require it. The low rank of the group is however important for the current proof of the Milnor-Wood inequality via semistability considerations for Higgs bundles.

As already mentioned above, in the singular case considering fundamental groups of varieties $X$ of general type is not the correct thing to do, as we have examples of quotients of smooth ball quotients by finite groups acting freely in codimension one that have trivial fundamental group (see \cite[Thm.~1.1]{keum}); for these examples the Toledo invariant defined via a representation of $\pi_1(X)$ is not maximal, that is, equality in the Milnor-Wood inequality for the fundamental group is no longer a necessary condition. 

\subsection*{Outline of the paper}
We will recall the notions and results about klt varieties, $G$-Higgs bundles over quasi-projective varieties and semistability results in \autoref{sec:back}.

We will then prove \autoref{thm:intro} in \autoref{sec:main}, by dividing the statement into the Milnor-Wood inequality and the uniformization result in the maximal case. 
To prove the main result we follow the strategy used in \cite{KMgentype} for the case of a smooth $X$. The Milnor-Wood inequality \autoref{thm:MWinequality} is proven via considering the Higgs bundle associated to the representations and using semistability results.
We will then show the uniformization result \autoref{thm:equality} in the maximal Toledo invariant situation via translating the problem to the case of a representation of the fundamental group of a smooth variety, making use of the existence of a maximally quasi-\'etale cover for klt varieties established in \cite{grebkebpetduke}. On this related smooth variety, we can then implement the same approach as in \cite{KMgentype} and conclude.
\par
\subsection*{Acknowledgments} We thank Vincent Koziarz for inspiring discussions and the referee for close reading and suggestions that improved the exposition of this paper.

	

\section{Foundational material}
\label{sec:back}

In this section we define the setting we are working in and recall notions and results about klt spaces and Higgs bundles over quasi-projective varieties. We finally state the two main ingredients that will be needed to prove the main inequality \eqref{thm:MWinequality} in the next section.

\subsection{Klt spaces }
\label{sec:klt}
Throughout the present paper, all varieties will be defined over the complex numbers. For a variety $X$  we will denote by $X_\reg$ the regular locus and by $i:X_\reg\to X$ the inclusion. We will also refer to \cite{kollarmori} for the definition of \emph{Kawamata log terminal singularities} and we will say that a variety $X$ is \emph{klt} if  it has at worst Kawamata log terminal singularities; in particular, we require $X$ to be  normal, and $K_X:=i_*(K_{X_\reg})$ to be a $\bQ$-Cartier divisor.

We will say that a klt variety is \emph{of general type} if the canonical divisor is big. 

\begin{rem}
While our definition of "general type" is not birationally invariant, it is appropriate in our context, since there exist finite quotients of smooth ball quotient surfaces that have finite quotient (klt) singularities  whose canonical divisor is ample but whose minimal resolution is not of general type, see again \cite[Thm.~1.1]{keum}. Since one of our aims is to characterize such ball quotients, these examples should be included in the definition of "general type". It should be mentioned that for \emph{canonical} singularities, there is no difference between the two definitions. 
\end{rem}

Since we will work with canonical models of klt varieties of general type with nef canonical bundle, we remind the reader of the following property. 

 \begin{reminder}
 \label{reminder}
If $X$ is a projective klt variety of general type with nef canonical bundle, by the Basepoint-Free Theorem \cite[Thm.~3.3]{kollarmori} the canonical bundle $K_X$ is semiample, i.e., there is a power of $K_X$  that is base-point free and then defines a birational morphism $q_X : X \to X_\can$ to a klt projective variety whose canonical divisor $K_{X_\can}$ is ample and such that $K_X = q_X^* K_{X_\can}$, see \cite[Lem.~2.30 or Prop.~3.51]{kollarmori}. The \emph{canonical model} $X_\can$ can be identified with $\Proj (R(K_X))$, where $R(K_X)$ is the \emph{canonical ring} of $X$, i.e., the section ring of $K_X$.
 \end{reminder}

 We now recall an important result of Takayama relating the fundamental group of $X$ and the one of its canonical model. 
 
\begin{theorem}[{\cite[Cor.~1.1(1)]{takayama}}]
\label{thm:takayama}
Let $X$ be a projective klt variety of general type with nef canonical divisor and let $q_X:X\to X_\can$ be the canonical morphism. Then, the induced homomorphism of fundamental groups $q_{X*}:\pi_1(X)\to \pi_1(X_\can)$ is an isomorphism.
\end{theorem}

In addition, we will often use the following comparison result for fundamental groups, especially with $V = X_\reg$.

\begin{rem}
\label{rem:codim2}
If $V$ is a smooth quasi-projective variety and $Z$ is a closed subset of codimension at least two, then the inclusion induces an isomorphism $\pi_1(V)\cong \pi_1(V\setminus Z)$ of fundamental groups. 
\end{rem}
 On the other hand, even if the set of singular points of a variety has high codimension, the difference between the fundamental group of the smooth locus and of the entire variety may be huge; for example, the Kummer surface associated with the product of two elliptic curves has trivial fundamental group while its regular locus has infinite fundamental group. 

We now recall the strategy used in \cite{grebkebpetduke} to pass from studying representations and Higgs bundles over a quasi-projective variety to studying ones over a projective variety, thus solving the above problem at least when it comes to linear representations.  For this we first introduce the following notion.

\begin{definition}[Quasi-étale Galois  morphisms]
 A morphism $\gamma:X \to Y$ between normal varieties is called \emph{quasi-étale} if it is of relative dimension zero and étale in codimension one, i.e. if $\dim(X) = \dim(Y)$ and if there exists a closed subset $Z\subseteq X$ of codimension at least two such that $\gamma_{|X\setminus Z}: X\setminus Z \to Y$ is étale. The morphism is moreover called \emph{Galois} if there exists a finite group $G \subset \Aut(X)$ such that $\gamma$ is isomorphic to the quotient map.
\end{definition}

With this terminology, we have the following useful result.
\begin{theorem}[{\cite[Thm.~1.5]{grebkebpetduke}}]
\label{thm:maximalquasietale}
Let $X$ be a  quasi-projective klt variety. Then, there exists a normal variety $Y$ and a finite, surjective quasi-étale Galois morphism $\gamma:Y\to X$ such that the natural map $\widehat{i}_*:\widehat{\pi}_1(Y_\reg)\to \widehat{\pi}_1(Y)$ of étale fundamental groups\footnote{Recall that the étale fundamental group is isomorphic to the profinite completion of the fundamental group of the underlying topological space (with the Euclidean topology).} induced by the inclusion of the smooth locus is an isomorphism. We call such a morphism a maximally quasi-étale covering.
\end{theorem}

As already mentioned above, the main property of the existence of maximally quasi-étale coverings used later is that it allows us to pass from a representation of the fundamental group of the regular locus of a klt variety to the one of a projective variety.

\begin{prop}[{\cite[Sect.~8.1]{grebkebpetduke}}]
\label{prop:factorization}
Let $\gamma:Y\to X$ be a maximally quasi-étale covering. Then any linear representation $\rho:\pi_1(Y_\reg)\to \GL_n(\bC)$ factors through a representation  $\rho_Y:\pi_1(Y)\to \GL_n(\bC)$.
\end{prop}
In fact,  the representation $\rho_Y$ is defined by using that $\widehat{i}_*:\widehat{\pi}_1(Y_\reg)\to \widehat{\pi}_1(Y)$ is an isomorphism and setting $\rho_Y:=\widehat{\rho}\circ \widehat{i}_*^{-1}\circ c$, where $c:\pi_1(Y)\to \widehat{\pi}_1(Y)$ and where $\widehat{\rho}$ denotes the pro-finite completion of $\rho$.

\subsection{$G$-Higgs bundles over quasi-projective varieties}
\label{sec:Ghiggs}
For a full treatment of $G$-Higgs bundles we refer for example to \cite{km2}, \cite{MR3382029} and \cite{mochizuki_asterisque}. Here we fix notation and recall the notions that are most relevant for the subsequent discussion.

Let $X$ be a variety, let $X^\mathfrak{u}$ be its universal covering,  and let $\rho:\pi_1(X)\to G$ be a representation into a linear reductive algebraic group over $\bR$. 
We will denote by $P_G$ the associated holomorphic flat $G$-principal bundle over $X$. Let $K$ be a maximal compact subgroup of $G$. We will denote by $P_K\subset P_G$ a $K$-reduction, i.e. a $\cC^{\infty}$-subbundle such that $P_K\times_K G\cong P_G$. Such a reduction is equivalent to a section of $P_G\times_G \cY_G$, where $\cY_G=G/K$ is the associated symmetric space, which is again equivalent to a map $f_\rho:X^\mathfrak{u}\to \cY_G$ since $P_G\cong X^\mathfrak{u}\times_\rho G $ is flat. Such a reduction is called \emph{pluri-harmonic} if the map $f_\rho$ is pluri-harmonic. 
 
Let now $X$, as above, be a klt variety and $X_\reg$ its regular locus.
We recall now the existence of pluri-harmonic equivariant harmonic maps in the case of reductive representations of the fundamental group of quasi-projective varieties.

\begin{theorem}[{\cite[Prop.~A.18]{mochizuki_asterisque}}]
 \label{thm:mochizuki}   
 Let $\rho:\pi_1(X_\reg)\to G$ be a reductive representation into a linear reductive algebraic group over $\bR$ or $\bC$. Then there exists a unique tame purely imaginary pluri-harmonic reduction $P_K\subset P_G$ and so there is the corresponding $\rho$-equivariant pluri-harmonic map $f_\rho:(X_\reg)^\mathfrak{u}\to \cY_G$. 
\end{theorem}

For any faithful complex representation $G\to \GL(\bE)$, we can construct the holomorphic flat vector bundle $E_\reg:=P_{K_\bC} \times_{K_\bC} \bE$ associated to $P_G$ over $X_\reg$. The pluriharmonicity of $f_\rho$ ensures that there is a holomorphic structure on $E_\reg$ and a holomorphic Higgs field $\theta_\reg:E_\reg\to E_\reg\otimes \Omega^1_{X_\reg}$, so that we can define the Higgs bundle $(E_\reg,\theta_\reg)$ over $X_\reg$. The Higgs field $\theta_\reg$ can be constructed as a holomorphic $(1,0)$-form taking values in $P_{K_\bC}\otimes_{\Ad K_\bC} \mathfrak{p}_\bC\cong f_\rho^*(T_\bC\cY_G)$ and can be identified with the derivative of $f_\rho$ by using that $\mathfrak{p}_\bC$ can be seen as a subspace of $\End(\bE)$. 

We will not recall the full definition of \emph{tame purely imaginary}, since we will not need it (see \cite[Def.~A.6]{mochizuki_asterisque} for a full definition). This is a condition on the growth of the harmonic metric near the
complement $X\setminus X_\reg$ ensuring a well-defined theory. 
\begin{rem}
\label{rem:tame}
    If $(E_\reg,\theta_\reg)$ is the Higgs bundle associated to the harmonic bundle $P_K\subset P_G$, then by \cite[Lem.~A.12]{mochizuki_asterisque} $(E_\reg,\theta_\reg)$ is tame purely imaginary (as defined in \cite[Def.~3.5-3.6]{grebkebpetmathann}).  We will use this definition only implicitly in \autoref{thm:higgspoly}, where we cite \cite[Thm.~1.2]{grebkebpetmathann}, for which this property is important. 
\end{rem}

 Let $H\in \Div(X)$ be a nef $\bQ$-Cartier divisor with numerical class $[H] \in N^1(X)_\mathbb{Q}$, which we call a \emph{polarization}. Then, since $X$ is normal and hence non-singular in codimension one, following for example \cite[Sect.~4.1]{grebkebpetduke} and extending it slightly from the case of Cartier divisors to $\mathbb{Q}$-Cartier divisors, we define the \emph{degree} and \emph{slope} of a vector bundle $E_\reg$ \emph{with respect to $H$} to be 
\[\deg_H(E_\reg):=c_1(i_*(E_\reg))\cdot [H]^{d-1},\quad \mu_H(E_\reg):=\frac{\deg_H(E_\reg)}{\rk(E_\reg)}.\]
Note that we will mostly work with (big) semiample divisors $H$, for which the intersection number can be computed by restricting the vector bundle $E_{\reg}$ to a high degree complete intersection curve lying in the smooth locus of $X$, see~\cite[Rem.~4.5]{grebkebpetduke} for more details. If the polarization $H$ is clear, we will omit it from the notation.

We recall now the notions of \emph{stability} for the Higgs bundle $(E_\reg,\theta_\reg)$.
\begin{definition}\label{def:stability}
    Let $H\in \Div(X)$ be a nef $\bQ$-Cartier divisor. We say that $(E_\reg,\theta_\reg)$ is \emph{stable} with respect to $H$ if any  generically Higgs-invariant subsheaf $F_\reg\subset E_\reg$ with $0<\rk(F_\reg)<\rk(E_\reg)$ satisfies $\mu_H(F_\reg)<\mu_H(E_\reg).$
    Analogously, we define notions of \emph{semistable} and \emph{polystable}.
\end{definition} 

Note that also in the case of a representations of fundamental groups of quasi-projective varieties, the Toledo invariant does not change under deformations induced by the $\bC^*$-action that scales the Higgs field, see \cite[Thm.~10.1, Sec. 10]{mochizuki_asterisque} and recall that the Toledo invariant takes rational values. Moreover, since in the limit the Higgs bundle has the structure of a system of Hodge bundles (\cite[Thm.~10.5]{mochizuki_asterisque}), as in the classical case we can reduce the computation to this situation. This is the approach used in \cite{KMgentype} and by the previous discussion there are no obstructions to applying the same strategy in the quasi-projective situation.

We finally recall the special situation where $G=\SU(p,q)$ following \cite{km2}.
If $\bE=\bV\oplus \bW$ is the vector space where $G$ acts, then $K=S(U(p)\times U(q))$ preserves the decomposition and so $E_\reg=P_{K_\bC} \times_{K_\bC} \bE$ splits holomorphically as the direct sum of the rank $p$ subbundle $V_\reg=P_{K_\bC} \times_{K_\bC} \bV$  and the rank $q$ subbundle $W_\reg=P_{K_\bC} \times_{K_\bC} \bW$. Since the Higgs field $\theta_\reg$ is a holomorphic $(1,0)$-form taking values in $f^*(T_\bC\cY_G)=\Hom(W_\reg,V_\reg)\oplus\Hom(V_\reg,W_\reg)$, we can write $\theta$ as 
\[\begin{pmatrix}
0 & \beta_\reg\\
\alpha_\reg & 0 
\end{pmatrix}, \quad \begin{cases}
    \beta_\reg:W_\reg\to V_\reg\otimes \Omega^1_{X_\reg} \\
    \alpha_\reg:V_\reg\to W_\reg\otimes \Omega^1_{X_\reg} 
\end{cases}\]

Since 
\begin{equation}
\label{eq:kahlerrelation}
c_1(K_{\cY_G})=\frac{p+q}{4\pi}\omega_{\cY_G}, \quad f^*(T^{1,0}\cY_G)=\Hom(W_\reg,V_\reg)
\end{equation}  
we can write the Toledo invariant as
\begin{equation}
\label{eq:toledoinv}
\tau(\rho)=\frac{\deg_{K_X}(f^*K_{\cY_G})}{p+q}=-\frac{p\deg_{K_X}(W_\reg^\vee)+q\deg_{K_X}(V_\reg)}{p+q}=\deg_{K_X}(W_\reg)
\end{equation} 
where the last equality follows from $c_1(i_*(E_\reg))=0$ stated in \autoref{thm:higgspoly}.

A similar analysis is possible for the group $G=\SO_o(p,2)$ (see \cite[Sec. 4.2]{km2}) and can be used for the proof of the main result in the case of representations of the Spin groups $\Spin(p,2)$ which we will not directly treat here.

\subsection{Semistability results}

The two main ingredients in order to prove the main inequality are semistability results.
The first result needed in \cite{KMgentype} is a  theorem by Enoki (\cite{Enoki}) about the semistability of the tangent bundle on good models of varieties of general type. In our context of klt varieties, we need a generalization of the previous result by Guenancia, or more specifically the following version by Greb-Kebekus-Peternell-Taji that only assumes $K_X$ to be nef.

\begin{theorem}[\cite{guenanciass}, {\cite[Thm.~7.1]{GKPT}}]
\label{thm:guenanciass}
Let $X$ be a projective, klt variety of general type whose canonical divisor $K_X$ is nef. Then the tangent sheaf $T_X$ is semistable with respect to $K_X$, when considered as a Higgs bundle with trivial Higgs field on $X_\reg$, cf.~Definition~\ref{def:stability}.
\end{theorem}

The second result needed in \cite{KMgentype} was shown in \cite[Lemma 4.2]{KMgentype}  and concerns the polystability of the Higgs bundle associated to a representation of the fundamental group of a smooth variety of general type.
This result can be replaced in our framework by the following.
\begin{theorem}
\label{thm:higgspoly}
Let $X$ be projective klt variety of general type with nef canonical bundle and of dimension greater than one. Let $\rho:\pi_1(X_\reg)\to G$ be a reductive representation into a linear reductive group. 
Then, the Higgs bundle $(E_\reg,\theta_\reg)$ on the regular locus associated to $\rho$ is $K_X$-polystable. Moreover $c_1(i_*(E_\reg))=0$, where $i:X_\reg\to X$ is the inclusion of the regular locus.
\end{theorem}
\begin{proof}
By Remark \ref{rem:tame}, the Higgs bundle associated to the representation is tame purely imaginary. Moreover, since by assumption the representation is reductive, the associated flat bundle is semisimple. If $K_X$ is ample, the claim was then shown in \cite[Thm. 1.2]{grebkebpetmathann}. If $K_X$ is only big and nef, then $K_X$ is semiample and it induces a birational morphism to the canonical model $q:X \to X_\can$, where $K_{X_\can}$ is ample and  $K_X = q^*(K_{X_\can})$, see Reminder \ref{reminder}. If we define $U=X_{\can,\reg}\setminus q(X_{\sing} \cup \mathrm{Exc}(q))$ to be the complement of the image of the singular locus of $X$ and the exceptional set of $q$ in the regular locus of $X_\can$, then $q$ induces an isomorphism between $U\subseteq X_{\can,\reg}$ and $q^{-1}(U)\subseteq X_\reg$, where the complement of the first subset has codimension at least two. By Remark \ref{rem:codim2},  we have the following natural identifications and morphisms
\[\pi_1(X_{\can,\reg}) \cong \pi_1(U) \cong \pi_1(q^{-1}(U)) \rightarrow \pi_1(X_\reg) ,\] where the last arrow is induced by the open inclusion $q^{-1}(U)\hookrightarrow X_\reg$.
Denote by $(E_{\can,\reg},\theta_{\can,\reg})$ the Higgs bundle on $X_{\can,\reg}$ induced by the composed map. Then $(E_\reg,\theta_\reg)\cong q^*(E_{\can,\reg},\theta_{\can,\reg})$ on $q^{-1}(U)$. Since $K_{X_\can}$ is ample, $(E_{\can,\reg},\theta_{\can,\reg})$ is $K_{X_\can}$-polystable and so the same is then true for its pull-back $(E_\reg,\theta_{\reg})$ with respect to $[K_X]=[q^*K_{X_\can}]$; indeed, a general high degree complete intersection curve for $K_{X_\can}$ will lie in $U$ and will be the isomorphic image of a general high degree complete intersection curve for $q^*K_{X_\can}$ contained in $q^{-1}(U)$. 
\end{proof}

	
	\section{The main result}
	\label{sec:main}

In this section we prove \autoref{thm:intro}. We will follow the proof of  the Milnor-Wood inequality in the case of representations of fundamental groups of smooth varieties of general type presented in \cite{KMgentype}.  As in \emph{loc.~cit.}, we will split the proof into the inequality part, which requires the rank of $G$ to be at most two, and the equality situation, which does not require the rank assumption.

\subsection{The inequality}

We show here the klt version of \cite[Prop.~4.3]{KMgentype}.
The two main ingredients for proving the inequality are the semistability results \autoref{thm:guenanciass} and \autoref{thm:higgspoly} for klt varieties. We will use the notation introduced at the end of Section~\ref{sec:Ghiggs}. 

\begin{theorem}
\label{thm:MWinequality}
Let $X$ be a   projective klt variety  of general type of dimension $d\geq 2$ with nef canonical divisor $K_X$. Let $G$ be either $\SU(p,q)$ with $1\leq q\leq 2\leq  p$, $\Spin(p,2)$ with $p\geq 3$ or $Sp(2,\bR)$. Finally let $\rho:\pi_1(X_\reg)\to G$ be a reductive representation. Then the Milnor-Wood type inequality
\begin{equation}\label{eq:specialMWinequ}
|\tau(\rho)|\leq \rk(G)\cdot \frac{K_{X}^d}{d+1}\end{equation}
holds.
Moreover, if $\tau(\rho)$ is maximal, then $G=\SU(p,q)$ with $p\geq dq$ and  we have $\deg(W_\reg)=-\deg(\Im(\beta_\reg))$ and $\beta_\reg:W_\reg\otimes T_{X_\reg}\to V_\reg$ is generically injective, where $E_\reg=W_\reg\oplus V_\reg$ is the splitting induced by the $\SU(p,q)$ structure and $\beta_\reg$ is one of the corresponding parts of the Higgs field. 
\end{theorem}

As remarked right after the main theorem in the introduction of \cite{KMgentype}, the statement is given in terms of groups whose complexification is simply connected, since in this case it is possible to define the Toledo invariant as the degree of a line bundle. However, \autoref{thm:MWinequality} implies the analogous result for all classical Hermitian Lie groups of rank $\rk(G)\leq 2$ apart from $G=\SO^*(10)$, as stated in the introduction, see \autoref{thm:intro}. This follows from \cite[Lemma 2.2]{KMgentype}, which implies that the main result is true for any representation in a quotient by a finite normal subgroup of any of the groups $G$ listed in \autoref{thm:MWinequality}, so for any classical Hermitian Lie groups of rank less than three apart from $G=\SO^*(10)$.

\begin{proof}[Proof of \autoref{thm:MWinequality}]
The proof proceeds along the same lines as the one in the smooth case \cite[Prop.~4.3]{KMgentype}, where we substitute the original ingredients by their respective generalizations, namely \autoref{thm:guenanciass} and \autoref{thm:higgspoly}. The proof uses a case-by-case analysis depending on the type of the group $G$. We recall how it works in the special case $G=\SU(p,1)$ to show how to use the generalized results in the klt setting, and refer to the original source for the full discussion.
\par
Recall from \autoref{sec:Ghiggs} that for $G=\SU(p,q)$ the Higgs bundle splits as $E_\reg=V_\reg\oplus W_\reg$ and the Higgs field splits as $\theta_\reg=\alpha_\reg\oplus \beta_\reg$. 
Moreover, by \eqref{eq:toledoinv} we have $\tau(\rho)=\deg(W_\reg)$, so in this case we need to show that 
\begin{equation}\label{eq:needtoshowInequality}
    \left|\deg_{K_X}(W_\reg)\right|\leq q\frac{K_{X}^d}{d+1}
\end{equation} We will omit the subscript when computing degrees for the remainder of this proof.

\par
Consider the case of a reductive representation with $G=\SU(p,1)$. Since $T_X$  is semistable by \autoref{thm:guenanciass} and $W_\reg$ a line bundle,  the tensor product $T_{X_\reg}\otimes W_\reg$ is semistable. Hence, interpreting $\beta_\reg$ as a map from $W_\reg\otimes T_{X_\reg}$ to $V_\reg$ we obtain a first inequality, \[\mu(\ker_{\beta_\reg})\leq \mu(W_\reg\otimes T_{X_\reg}).\] Moreover, since $ \Im(\beta_\reg) \oplus W_\reg$ is a  Higgs subsheaf of $(E_\reg,\theta_\reg)$, semistability of the latter Higgs bundle, \autoref{thm:higgspoly}, implies \[\deg(W_\reg\oplus \Im(\beta_\reg))\leq 0.\] Putting these preliminary considerations together we obtain 
\begin{align*}
\deg(W_\reg\otimes T_{X_\reg})&=\deg(\Im(\beta_\reg)) + \deg(\ker_{\beta_\reg}) \\
&\leq \frac{\rk(\ker_{\beta_\reg})}{d}\deg(W_\reg\otimes T_{X_\reg})-\deg(W_\reg).
\end{align*}
Noting that both $W_\reg$ and $T_{X_\reg}$ are locally free along any complete intersection curve for $K_X$ used to compute the degrees, so that tensor products work as expected, we get
\[\deg(W_\reg)\leq \frac{d-\rk(\ker_{\beta_\reg})}{d-\rk(\ker_{\beta_\reg})+1}\frac{\deg(K_X)}{d}.\]
Using finally that $\rk(\Im_{\beta_\reg})=d-\rk(\ker_{\beta_\reg})\leq d$, we arrive at the upper bound of the desired inequality \eqref{eq:needtoshowInequality} (recall that we are in the special case where $q=1$).
Moreover, equality is possible only if $\rk(\ker_{\beta_\reg})=0$ and $\deg(W_\reg\oplus \Im(\beta_\reg))=0$; furthermore, $\rk(\ker_{\beta_\reg})=0$ implies the restriction $p\geq qd$ imposed in the statement of \autoref{thm:MWinequality}.
\par
For $G=\SU(p,2)$ the arguments are similar, but they have to be applied to a deformation of $E_\reg$ given by a system of Hodge bundles, which can be found as a fixed point of  the $\bC^*$-action that scales the Higgs field $\theta_\reg$.

This also settles the case $G=\Sp(2,\bR)\subset \SU(2,2)$.  Indeed,  maximal representations are impossible for $G=\Sp(2,\bR)$, since in this case the inequality $p\geq qd$ derived in a previous step cannot be achieved with $d\geq 2$. In case $G=\Spin(p,2)$, the situation is similar, that is, deforming to a system of Hodge bundles one obtains a stronger bound that makes the maximal situation impossible. We refer to \cite{km2} for the full treatment of the previously mentioned cases.

Note finally that the proof of the lower bound of the Milnor-Wood inequality \autoref{eq:specialMWinequ} follows using the same strategies by considering the dual of the Higgs bundle $(E_\reg,\theta_\reg)$. 

\end{proof}

\begin{rem}[Nonreductive representations]\label{rem:nonred}
In principle, the Toledo invariant can be defined for an arbitrary representation $\rho: \pi_1(X_{\text{reg}})\longrightarrow G$ into a Hermitian Lie group, by using in \eqref{eq:MWinequ} any smooth $\rho$-equivariant map $f: (X_{\text{reg}})^\mathfrak{u}\to \mathcal{Y}_G$ arising from the choice of a background metric and the representation $\rho$.  Since any two such maps are homotopic, the Toledo invariant does not depend on this choice. In case $\rho$ is reductive, we may use the harmonic map coming from Mochizuki's work, and hence for reductive representations this produces the same number $\tau(\rho)$ as before. Moreover, as explained in \cite[Sect.~3.3.3 / 4.2]{km2} the inequality \eqref{eq:specialMWinequ} for non-reductive representations follows from the inequality for reductive ones, while equality is never achieved for non-reductive representations into the groups considered here; in particular, if $G= \mathrm{PU}(p,1)$, then nonreductive representations have vanishing Toledo invariant,  see~\cite[Lem.~3.4]{km0}.
\end{rem}

\subsection{The equality case}

In this section we will show the uniformization statement of our main result, Theorem~\ref{thm:intro}. Again, this leads to the two cases of $\tau(\rho)$ being maximal or minimal. We will deal with the first case in detail by proving the subsequent \autoref{thm:equality}, which is the equivalent of \cite[Prop.~4.4]{KMgentype} and takes into account the information provided by \autoref{thm:MWinequality}. We will comment on the changes necessary to cover the second case directly after the proof, thus concluding the proof of \autoref{thm:intro}. 

\begin{theorem}
\label{thm:equality}
Let $X$ be a  projective klt variety  of general type with nef canonical bundle of dimension $d\geq 2$, and let $X_{\text{reg}}\subseteq X$ be its regular locus. Let $\rho:\pi_1(X_\reg)\to \SU(p,q)=G$ with $p\geq qd$ and $q\geq 1$ and $(E_\reg=V_\reg\oplus W_\reg,\theta_\reg=\beta_\reg\oplus \alpha_\reg)$ be its associated Higgs bundle. Assume that $\tau(\rho)$ is maximal, that is, $\deg(W_\reg)=q\frac{K_{X}^d}{d+1}$, that we have $\deg(W_\reg)=-\deg(\Im(\beta_\reg))$, and that $\beta_\reg:W_\reg\otimes T_{X_\reg}\to V_\reg$ is generically injective. Then, the canonical model $X_{\can}$ of $X$ is the quasi-étale quotient of a smooth ball quotient $Z$ by a finite group. Moreover,  the induced $\rho_Z$-equivariant harmonic map $f_Z:Z^\mathfrak{u}\to \cY_G$ is a holomorphic  proper embedding from the universal cover of $Z$  onto a totally geodesic copy of complex hyperbolic $d$-space.
\end{theorem}

In order to analyze the equality case, we use information on how the harmonic metric on $X_\reg$ arises and compares to such metrics on associated covers and resolutions, cf. the proof of \cite[Thm.~1.2]{grebkebpetmathann}. The situation is summarized by the following diagram, whose notation will be discussed in the subsequent paragraph: 

\begin{center}
\begin{tikzcd}\label{bigdiagram}
&\cY_G&\widetilde{Y}_\can^\mathfrak{u}\arrow{d}{\widetilde{\eta}} \arrow[swap]{l}{f_{\rho_{\widetilde{Y}_\can}}}\\
(X_\reg)^\mathfrak{u} \arrow[bend right=60]{dd} \arrow{d}{\eta} \arrow{ur}{f_\rho} &&\widetilde{Y}_\can \arrow{d}{\varphi}\\
 \gamma^{-1}(X_\reg)\arrow{r}{\subset}\arrow{d}{\gamma} &Y \arrow{d}{\gamma} \arrow{r}{q_Y}& Y_\can\\
X_{\reg}\arrow{r}{\subset} & X&
\end{tikzcd}
\end{center}

In the above diagram, the map $\gamma\colon Y\to X$ is a maximally quasi-étale cover, i.e., a finite surjective quasi-étale Galois morphism  such that the natural map $\widehat{i}_*:\widehat{\pi}_1(Y_\reg)\to \widehat{\pi}_1(Y)$ is an isomorphism. The existence of this map is guaranteed by \autoref{thm:maximalquasietale}. Since $\gamma$ is unbranched over $X_\reg$, the universal cover of $X_\reg$ factorizes through $\gamma^{-1}(X_\reg) \subset Y$. Note that $Y$ is a klt variety of general type with nef canonical bundle, so $K_Y$ is again semiample. We denote by $Y_\can$ its canonical model, by $\varphi:\widetilde{Y}_\can\to Y_\can$ a strong resolution of singularities of $Y_\can$ and by $\widetilde{Y}_\can^\mathfrak{u}$ the universal cover of the resolution. As the singularities can be worse than canonical, note that $\widetilde{Y}_\can$ usually will not be of general type, see for example \cite[Thm.~1.1]{keum}, which describes finite quasi-étale quotients of fake projective planes having resolution of Kodaira dimension one. On the other hand, $\varphi^*K_{Y_\can}$ is still big and nef and so can be used as a polarization on $\widetilde{Y}_\can$. Moreover, differently from the situation of \cite{KMgentype}, $\varphi$ is not the canonical morphism of $\widetilde{Y}_\can$; however, we will see that the strategy of \emph{loc.~cit.} also works in this setting, once suitably modified.

The following result allows us to pass from the situation of a representation of the fundamental group of the regular locus of a klt variety to a representation of the fundamental group of a smooth compact variety. 
\begin{prop}
\label{prop:MWres}
The representation $\rho:\pi_1(X_\reg)\to G$ naturally induces a representation $\rho_{\widetilde{Y}_\can}:\pi_1(\widetilde{Y}_\can)\to G$ such that 
\[\int_X f_\rho^*(\omega_{\mathcal{Y}_G})\wedge \c_1(K_X)^{d-1}= \frac{1}{\deg(\gamma)}\int_{\widetilde{Y}_\can} f_{\rho_{\widetilde{Y}_\can}}^*(\omega_{\mathcal{Y}_G})\wedge \c_1(\varphi^*K_{Y_\can})^{d-1}.\]
\end{prop}
\begin{proof}
Since $\gamma$ is étale in codimension 1, the representation $\rho$ induces a representation $\gamma^*\rho:\pi_1(\gamma^{-1}(X_\reg))\to G$, which by Remark \ref{rem:codim2} extends to a representation of $\pi_1(Y_\reg)$. Recall now from \autoref{prop:factorization} that, since by construction there is an isomorphism $\hat{\pi}_1(Y_\reg)\to \hat{\pi}_1(Y)$ of étale fundamental groups, any representation of $\pi_1(Y_\reg)$ factors through a representation of $\pi_1(Y)$. Hence $\rho$ induces a representation $\rho_Y:\pi_1(Y)\to G$ that factorizes the representation $\gamma^*\rho$ induced by $\gamma$. By Takayama's theorem,  \autoref{thm:takayama}, the canonical map $q_Y$ induces an isomorphism of fundamental groups $\pi_1(Y)\cong \pi_1(Y_\can)$, which we can use to induce from $\rho_Y$  a representation $\rho_{Y_\can}:\pi_1(Y_\can)\to G$. We finally denote by $\rho_{\widetilde{Y}_\can}:=\varphi^*\rho_{Y_\can}$ the representation obtained as pull-back via the resolution $\varphi$, which similarly to $q_Y$ induces an isomorphism on the level of fundamental groups.

Since $\gamma$ is quasi-étale  and since $q_Y$ is the canonical map, we have the following equalities that lead to analogous equalities for the first Chern classes:
 \begin{equation}
     \label{eq:canonicalcompatibility}
     K_Y = \gamma^*(K_X) = q_Y^*(K_{Y_\can}).
 \end{equation}

Let us introduce the following collection of open subsets:
\begin{align*}&U_{\can} := Y_{\can, \reg} \setminus q_Y(\gamma^{-1}(X_\sing) \cup \mathrm{Exc}(q_Y)),\\
&U_Y:= (q_Y)^{-1}(U_\can) \subset \gamma^{-1}(X_\reg), \\
&\widetilde{U} := \varphi^{-1}(U_\can),\\&U_X:=\gamma(U_Y).
\end{align*}
Note that owing to \eqref{eq:canonicalcompatibility} the map $q_Y$ is $\mathrm{Gal}(Y/X)$-equivariant, the subset $U_Y$ is $\mathrm{Gal}(Y/X)$-invariant and therefore equal to $\gamma^{-1}(U_X)$ and that it is isomorphic to $U_\can$ via $q_Y$ by definition. Furthermore, note that $U_\can$  and $\widetilde U$ are isomorphic via $\varphi$.  Moreover, since $\pi_1(Y)$ and $\pi_1(Y_\can)$ are isomorphic via $(q_Y)_*$, and  $\pi_1(Y_\can)$ and $\pi_1(\widetilde{Y}_\can)$ are isomorphic via $\varphi_*$, also the map $(\varphi^\mathfrak{u})^{-1}\circ q_Y^\mathfrak{u}:\eta^{-1}(U_Y)\overset{\cong}{\to} \widetilde{\eta}^{-1}(\widetilde{U})$
induced on the lifts of $U_Y$ and $\widetilde{U}$ to the respective universal coverings is an isomorphism.
Since by construction \[\rho_{\widetilde{Y}_\can}=\rho \circ \gamma_* \circ (q_Y)_{*}^{-1}\circ \varphi_*,\] the uniqueness of the corresponding harmonic equivariant maps implies
 \begin{equation}
 \label{eq:Toledocompatibility}
 f_{\rho_{\widetilde{Y}_\can}}(p) ={f_\rho} \circ\left((\varphi^\mathfrak{u})^{-1}\circ q_Y^\mathfrak{u}\right)^{-1} (p) \quad \text{for all }p \in {\widetilde{\eta}^{-1}(\widetilde{U})}.
 \end{equation}

 From the previous observations we deduce the following chain of equalities
 \begin{align*}
 \deg(\gamma)\cdot\int_{U_X} f_\rho^*(\omega_{\mathcal{Y}_G})\wedge \c_1(K_X)^{d-1}&=\int_{U_Y} \gamma^*\left(f_\rho^*(\omega_{\mathcal{Y}_G})\wedge \c_1(K_X)^{d-1}\right) \\
 &=\int_{U_Y} \gamma^*\left(f_\rho^*(\omega_{\mathcal{Y}_G})\right) \wedge q_Y^*\c_1(K_{Y_\can})^{d-1}\\
 &=\int_{U_\can} (q_Y^{-1})^*\gamma^*\left(f_\rho^*(\omega_{\mathcal{Y}_G})\right) \wedge \c_1(K_{Y_\can})^{d-1}\\
&=\int_{\widetilde{U}} \varphi^*\left((q_Y^{-1})^*\gamma^*\left(f_\rho^*(\omega_{\mathcal{Y}_G})\right) \wedge \c_1(K_{Y_\can})^{d-1}\right)\\
&=\int_{\widetilde{U}} f_{\rho_{\widetilde{Y}_\can}}^*(\omega_{\mathcal{Y}_G})\wedge \c_1(\varphi^*K_{Y_\can})^{d-1}.
 \end{align*}
 Indeed, the first equality holds, since $\gamma$ is a maximally quasi-étale morphism and owing to the fact that $U_Y=\gamma^{-1}(U_X)$, as noticed above. The second equality follows from \eqref{eq:canonicalcompatibility}, the third and fourth ones use the fact that $U_Y$ and $U_\can$ are isomorphic via $q_Y$, that $U_\can$  and $\widetilde U$ are isomorphic via $\varphi$, and that the last equality follows from \eqref{eq:Toledocompatibility}.
 
Note finally that, since $K_{Y_\can}$ is ample, we can choose a smooth representative of the curve class $\c_1(K_{Y_\can})^{d-1}$ completely contained in $U_\can$. This also means that we can do the same for $\varphi^*\c_1(K_{Y_\can})^{d-1}$ in $\widetilde{U}$ and for $q_Y^*\c_1(K_{Y_\can})^{d-1}=\gamma^*\c_1(K_{X})^{d-1}$ in $U_Y$. Owing to $U_Y=\gamma^{-1}(U_X)$, we can then finally also choose a smooth representative of the curve class $\deg(\gamma)\c_1(K_{X_\can})^{d-1}$ completely contained in $U_X$, and hence all the previous integrals are indeed  well-defined and extend to the respective full spaces. This concludes the proof.
\end{proof}

\begin{cor}
    \label{cor:compatibilities}
    Let $X$ and $\rho$ as in \autoref{thm:equality} and $\rho_{\widetilde{Y}_\can}$ as above. Then the Higgs bundle $(V_{\widetilde{Y}_\can}\oplus W_{\widetilde{Y}_\can},\alpha_{\widetilde{Y}_\can}\oplus\beta_{\widetilde{Y}_\can})$ on $\widetilde{Y}_\can$ induced by $\rho_{\widetilde{Y}_\can}$ satisfies 
    \[\deg(W_{\widetilde{Y}_\can})=q\frac{K_{Y_\can}^d}{d+1},\quad \deg(W_{\widetilde{Y}_\can})=-\deg(\Im(\beta_{\widetilde{Y}_\can}))\] and  $\beta_{\widetilde{Y}_\can}:W_{\widetilde{Y}_\can}\otimes T_{\widetilde{Y}_\can}\to V_{\widetilde{Y}_\can}$ is generically injective.
\end{cor}
\begin{proof}
    By the discussion around \eqref{eq:kahlerrelation} and \autoref{prop:MWres}, we have 
    \[\deg W_\reg=\tau(\rho)=\frac{1}{\deg(\gamma)}\int_{\widetilde{Y}_\can} f_{\rho_{\widetilde{Y}_\can}}^*(\omega_{\mathcal{Y}_G})\wedge \c_1(\varphi^*K_{Y_\can})^{d-1}=\frac{\deg(W_{\widetilde{Y}_\can})}{\deg(\gamma)}.\]
    This implies, by the maximality of $\rho$ and the relations \eqref{eq:canonicalcompatibility} among the canonical bundle, that  
    \[\deg(W_{\widetilde{Y}_\can})=\deg(\gamma)\left(q\frac{K_{X}^d}{d+1}\right)=q\frac{(\gamma^*K_{X})^d}{d+1}=q\frac{(q_Y^*K_{Y_\can})^d}{d+1}=q\frac{K_{Y_\can}^d}{d+1}.\]
    
    For the other two claims we can use  the compatibility \eqref{eq:Toledocompatibility} of the harmonic maps and the fact that $\beta_\reg$, resp. $\beta_{\widetilde{Y}_\can}$, is the $(1,0)$-component of $df_\rho$, resp. $df_{\rho_{\widetilde{Y}_\can}}$. Then the same arguments as the ones in the proof of \autoref{prop:MWres} yield $\deg(\Im(\beta_\reg))=\deg(\Im(\beta_{\widetilde{Y}_\can}))$. Moreover the previous comment also implies that, since $\gamma$ is quasi-étale and $q_Y$ and $\varphi$ are birational maps, the generic injectivity of $\beta_\reg$ is equivalent to the generic injectivity of $\beta_{\widetilde{Y}_\can}$.
\end{proof}

At this point, the strategy is to follow the proof of \cite[Prop.~4.4]{KMgentype} for the representation $\rho_{\widetilde{Y}_\can}$ of the full fundamental group of a compact smooth variety. Note that, as remarked above, the smooth variety $\widetilde{Y}_\can$ is not of general type in general, so we cannot apply \cite[Prop.~4.4]{KMgentype} directly, but we can prove the following variation.

\begin{prop}
\label{prop:coverball}
Let $Z$ be a normal projective variety  of dimension $d\geq 2$ such that $K_Z$ is ample and let $\varphi:\widetilde{Z}\to Z$ be a resolution of singularities. Let moreover $\rho_Z:\pi_1(Z)\to \SU(p,q)$ be a representation with $p\geq qd$ and $q\geq 1$ and let $\rho:=\varphi^*\rho_Z:\pi_1(\widetilde{Z})\to \SU(p,q)$ be the pull-back representation. Let $(E=V\oplus W,\theta=\beta\oplus \alpha)$ be the Higgs bundle on $\widetilde{Z}$ associated to $\rho$. Assume that $\tau(\rho)$ is maximal with respect to the polarization $\varphi^*K_Z$, i.e. $\deg(W)=q\frac{K_{Z}^d}{d+1}$, that we have $\deg(W)=-\deg(\Im(\beta))$ and that $\beta:W\otimes T_{\widetilde{Z}}\to V$ is generically injective. Then, $Z$ is a smooth ball quotient. Moreover, the $\rho_Z$-equivariant harmonic map $f_Z:Z^\mathfrak{u}\to \cY_G$ is a holomorphic  proper embedding from the universal cover of $Z$  onto a totally geodesic copy of complex hyperbolic $d$-space.
\end{prop}
\begin{proof}
We will follow the proof of \cite[Prop.~4.4]{KMgentype}.
As shown in \emph{loc.~cit.}, the fact that $\beta$ is generically injective implies that the harmonic map $f$ associated to $\rho$ is not only a generic immersion, but in fact holomorphic. Moreover, if we denote by $V'$ the saturation of $\Im(\beta)$ in $V$, we have that $\det \beta:\det(W)^d\otimes K_{\widetilde{Z}}^{-q}\to \det V'$ is generically injective, and so there exists an effective divisor $D$ such that \[\det(W)^d\otimes K_{\widetilde{Z}}^{-q}\otimes \mathcal{O}(D)\cong \det V'.\] Since $\rho$ is maximal, we must have $\deg(D)=\int_D \varphi^*c_1(K_Z)^{d-1}=0$. This in turn implies that the support of $D$ has to be contained in the exceptional set $\Ex(\varphi)$ of $\varphi$, since $K_Z$ is ample. In summary, up to this point we have shown that $f:\widetilde{Z}^\mathfrak{u}\to \cY_G$ is holomorphic and an immersion outside the lift of $\Ex(\varphi)$ to $\widetilde{Z}^\mathfrak{u}$.

Using the argument presented on \cite[pp.~226f]{KMgentype}, one shows that the image of $f$ lies in a totally geodesic submanifold $\cZ$, where $\cZ\cong \bB^d$ is  a ball of maximal possible holomorphic sectional curvature in $\cY_G$. Note that the maximality of $\rho$ as a representation in $\SU(p,q)$ implies the maximality  as a representation in $\Aut(\cZ)$ since $\omega_\cZ=\frac{1}{q}{\omega_{\cY_G}}|_{\cZ}$.

Since $\rho$ is defined as the pull-back of the representation $\rho_Z$, the restriction of the holomorphic map $f$ to $(\varphi^{-1}(Z_\reg))^\mathfrak{u}$ factors through the universal cover of the regular locus of $Z$. Since $Z$ is normal, the singular locus has codimension at least two and we can extend the map to $\rho_Z$-equivariant holomorphic map $f_Z:Z^\mathfrak{u}\to \cZ$. Since $f_Z^*K_\cZ$ descends to $Z$ by equivariance, on the regular locus of $Z$ we have $K_Z=f_Z^*K_\cZ+R$, where $R$ is the zero locus of the Jacobian of $f_Z$. Again by normality of $Z$, the previous relation can be extended to the full of $Z$.

By the maximality of $\rho$ and by \eqref{eq:kahlerrelation} we have
\[K_Z^d=f^*K_\cZ\cdot (\varphi^*K_Z)^{d-1}=f_Z^*K_\cZ\cdot K_Z^{d-1}=K_Z^{d}-R\cdot K_Z^{d-1}\]
which implies $R=0$, since $R$ is effective and $K_Z$ is ample. Hence $f_Z$ is a local biholomorphism on $Z_\reg$. As explained at the end of the proof of \emph{loc.~cit.},  the purity of the ramification locus implies that $f_Z$ is a local biholomorphism on the full of $Z$ and so $Z$ is smooth and can be endowed with a metric of constant holomorphic sectional curvature $-1$, i.e. it is uniformized by the ball $\bB^d$.
\end{proof}

The last ingredient needed to prove \autoref{thm:equality} is the following lemma.

\begin{lemma}
\label{lem:canonicalmodels}
Let $\gamma:Y\to X$ be a quasi-étale Galois cover of projective klt varieties with big and nef canonical divisors. Then, the canonical model $X_\can$ of $X$ is a quasi-étale finite quotient of $Y_\can$.
\end{lemma}
\begin{proof}
Recall that the canonical model of a variety $Z$ can be defined as the Proj of its canonical ring $R(K_Z)$. Since  $\gamma\colon Y\to X$ is a quasi-étale Galois cover, i.e., in particular it is the quotient map by a finite group $G\subseteq \Aut(Y)$, we have $H^0(X,mK_X)\cong H^0(Y,mK_Y)^G$  for all $m\in \mathbb{N}_0$, where the right hand side is the ring of $G$-invariant sections. 
This implies that there is an induced map 
\[\gamma_\can:\Proj \left(R(K_Y)\right)\to \Proj \left(R(K_Y)^G\right)\cong \Proj \left(R(K_X)\right)\]
between the canonical models which is finite and such that $\gamma_\can\circ q_Y=q_X\circ\gamma$, where $q_X$ and $q_Y$ are the canonical morphisms cf.~Reminder~\ref{reminder}.
The previous commutativity implies that if $\gamma$ is quasi-étale, also $\gamma_\can$ is, since if we could find a divisor $D\subset Y_\can$ along which $\gamma_\can$ ramifies, then $\gamma$ would not be étale along the strict transform of $D$.
\end{proof}

We are finally able to prove the main uniformization result.

\begin{proof}[Proof of \autoref{thm:equality}]
  It is enough to put together all the ingredients we have shown above. Thanks to \autoref{cor:compatibilities}, we can apply \autoref{prop:coverball} with $Z=Y_\can$ and $\widetilde{Z}=\widetilde{Y}_\can$. This gives that $Y_\can$ is a smooth ball quotient, and that its associated harmonic map is a totally geodesic proper holomorphic embedding. Moreover, by \autoref{lem:canonicalmodels} we have that $X_\can$  is a quasi-étale finite quotient of the smooth ball quotient $Z=Y_\can$.
\end{proof}
As announced at the beginning of the current section, to conclude the proof of \autoref{thm:intro}, let us finally comment on how to deal with the case of minimal $\tau(\rho)$.

Using \autoref{prop:MWres} and \autoref{cor:compatibilities}, as above we induce a representation of $\pi_1(Y_\can)$  with minimal Toledo invariant. Then, we change the complex structure on $Y_\can$ to its opposite, i.e., we consider the variety $\overline{Y_\can}$. The fundamental group stays the same, but the Toledo invariant becomes maximal due to $\c_1(K_{\overline{Y_\can}})=-\c_1(K_{Y_\can})$. So, we are back in the situation considered before and conclude that $\overline{Y_\can}$ is a smooth ball quotient and that the associated harmonic map is a totally geodesic proper holomorphic embedding. This implies that the same holds for $Y_\can$ apart from the fact that its associated harmonic map is now antiholomorphic.  Moreover, as before, \autoref{lem:canonicalmodels} implies that $X_\can$  is a quasi-étale finite quotient of the smooth ball quotient $Y_\can$, as claimed.

The converse of the equality part of the statement of \autoref{thm:intro} follows from the corresponding statement in the main Theorem of \cite{KMgentype}.

	\vspace{0.4cm}
	\printbibliography
	
	\vfill
\end{document}